\documentclass[12pt]{article}
\usepackage{a4wide}
\usepackage{amsthm}
\usepackage{amsfonts}
\usepackage{amssymb}
\usepackage{amsmath}
\usepackage{cite}
\usepackage{epsfig}
\usepackage{stmaryrd}
\usepackage{xcolor}
\newtheorem{theorem}{Theorem}
\newtheorem{lemma}[theorem]{Lemma}
\newtheorem{proposition}[theorem]{Proposition}

\newcommand\FF{{\mathbb F}}
\DeclareTextCompositeCommand{\v}{OT1}{l}{l\nobreak\hspace{-.1em}'}
\DeclareTextCompositeCommand{\v}{OT1}{t}{t\nobreak\hspace{-.1em}'\nobreak\hspace{-.15em}}

\begin{document}
\title{Hypergraphs with uniform Tur\'an density equal to 8/27\thanks{The work of all authors was supported by the MUNI Award in Science and Humanities (MUNI/I/1677/2018) of the Grant Agency of Masaryk University.}}

\author{Frederik Garbe\thanks{Institute of Computer Science, Czech Academy of Sciences, Pod Vod\'arenskou v\v{e}\v{z}\i\i{} 271/2, 182 00 Prague, Czech Republic. Email: {\tt garbe@cs.cas.cz}. Previous affiliation: Universit\"at Heidelberg, Im Neuenheimer Feld 205, 69120 Heidelberg, Germany.}\hskip 0.75ex\thanks{Previous affiliation: Faculty of Informatics, Masaryk University, Botanick\'a 68A, 602 00 Brno, Czech Republic.}\and
        \newcounter{muni}
	\setcounter{muni}{3}
        Daniel I\v{l}kovi\v{c}\thanks{Institute of Mathematics, Leipzig University, Augustusplatz 10, 04109 Leipzig. E-mail: {\tt \{daniel.ilkovic,daniel.kral\}@uni-leipzig.de}.}\hskip 0.75ex\thanks{Mathematics Institute and DIMAP, University of Warwick, CV4 7AL Coventry, UK.}\hskip 1.00ex{}$^\fnsymbol{muni}$\and
        \newcounter{LU}
	\setcounter{LU}{4}
        Daniel Kr{\'a}\v l$^\fnsymbol{LU}$\thanks{Max Planck Institute for Mathematics in the Sciences, Inselstra{\ss}e 22, 041 03 Leipzig, Germany. E-mail: {\tt \{filip.kucerak,dan.kral\}@mis.mpg.de}.}\hskip 0.75ex{}$^\fnsymbol{muni}$\and
        \newcounter{mpi}
	\setcounter{mpi}{6}
	Filip Ku\v cer\'ak$^\fnsymbol{mpi}${}$^\fnsymbol{muni}$\and
        Ander Lamaison\thanks{Extremal Combinatorics and Probability Group (ECOPRO), Institute for Basic Science (IBS), Daejeon, South Korea. This author was also supported by IBS-R029-C4. E-mail: {\tt ander@ibs.re.kr}.}\hskip 1.75ex{}$^\fnsymbol{muni}$
	}

\date{}

\maketitle

\begin{abstract}
In the 1980s,
Erd\H{o}s and S\'os initiated the study of Tur\'an problems with a uniformity condition on the distribution of edges:
the uniform Tur\'an density of a hypergraph $H$ is the infimum over all $d$ for which
any sufficiently large hypergraph with the property that all its linear-size subhypergraphs have density at least $d$ contains $H$.
In particular, they asked to determine the uniform Tur\'an densities of $K_4^{(3)-}$ and $K_4^{(3)}$.
After more than 30 years,
the former was solved in [Israel J. Math. 211 (2016), 349--366] and [J. Eur. Math. Soc. 20 (2018), 1139--1159],
while the latter still remains open.
Till today, there are known constructions of $3$-uniform hypergraphs with uniform Tur\'an density
equal to $0$, $1/27$, $4/27$ and $1/4$ only.
We extend this list by a fifth value:
we prove an easy to verify sufficient condition for the uniform Tur\'an density to be equal to $8/27$ and
identify hypergraphs satisfying this condition.
\end{abstract}

\section{Introduction}

Tur\'an problems, which ask for the minimum density threshold for the existence of a certain substructure,
are one of the most fundamental kind of problems in extremal combinatorics.
Formally, the Tur\'an density of a graph $G$ is the infimum over all $d$ such that
every sufficiently large $n$-vertex graph with $d\binom{n}{2}$ edges contains $G$ as a subgraph.
The classical theorems of Mantel~\cite{Man07} and Tur\'an~\cite{Tur41} determine the Tur\'an density of complete graphs.
For general graphs,
Erd\H os and Stone~\cite{ErdS46} proved that
the Tur\'an density of any $r$-chromatic graph is equal to $\frac{r-2}{r-1}$, see also~\cite{ErdS66}.
While the situation is well-understood in the graph setting,
the analogous questions in the hypergraph setting are among the most challenging problems in extremal combinatorics.
Erd\H os~\cite{Erd81} offered \$1\,000 for determining the Tur\'an density of all complete $k$-uniform hypergraphs for $k\ge 3$ and
\$500 for determining the Tur\'an density of any single complete $k$-uniform hypergraph (with at least $k+1$ vertices).
Even determining the Tur\'an density of $K_4^{(3)}$, the complete $3$-uniform hypergraph with $4$ vertices,
has resisted all attempts for its resolution since its formulation 80 years ago~\cite{Tur41};
we refer to~\cite{FraF84,ChuL99,Raz10} for partial and related results, and
also to the surveys by Keevash~\cite{Kee11} and Sidorenko~\cite{Sid95}.

Most of the known and conjectured extremal constructions for Tur\'an problems have large independent sets,
i.e., linear-size sets of vertices with no edges.
This led Erd\H os and S\'os~\cite{ErdS82,Erd90} to propose studying Tur\'an problems
with the additional requirement that the edges of the host hypergraph are distributed uniformly.
Formally,
the \emph{uniform Tur\'an density} of a $k$-uniform hypergraph $H$
is the infimum over all $d$ such that for every $\varepsilon>0$,
there exists $n_0$ such that the following holds:
every $k$-uniform hypergraph $H_0$ with $n\ge n_0$ vertices such that
any subset of $n'\ge\varepsilon n$ vertices of $H_0$ spans at least $d\binom{n'}{k}$ edges
contains $H$ as a subgraph.
So, unlike in the case of Tur\'an problems,
uniform Tur\'an problems require host hypergraphs to possess the required edge density on all linear-size vertex subsets.
We remark that the notion is interesting for $k$-uniform hypergraphs with $k\ge 3$ only,
since it can be shown that the uniform Tur\'an density of every graph is equal to $0$ (this follows e.g.~from~\cite[Theorem 1]{Rod86}).

Erd\H os and S\'os also asked to determine the uniform Tur\'an density of the two smallest non-trivial $3$-uniform hypergraphs:
the complete $3$-uniform hypergraph $K_4^{(3)}$ with four vertices and
the $3$-uniform hypergraph $K_4^{(3)-}$, which is the hypergraph $K_4^{(3)}$ with an edge removed.
Determining the uniform Tur\'an density of $K_4^{(3)}$ remains a challenging open problem 
though it is believed that a 35-year-old construction of R\"odl~\cite{Rod86} showing that
the uniform Tur\'an density of $K_4^{(3)}$ is at least $1/2$ is optimal~\cite{Rei20}.
On a positive note,
the uniform Tur\'an density of $K_4^{(3)-}$ was shown to be equal to $1/4$
by Glebov, Volec and the third author~\cite{GleKV16} using arguments based on the flag algebra method of Razborov~\cite{Raz07} and
by Reiher, R\"odl and Schacht~\cite{ReiRS18a} using direct combinatorial arguments.

In addition to $K_4^{(3)-}$, there are very restricted families of $3$-uniform hypergraphs
whose uniform Tur\'an density is known exactly.
Reiher, R\"odl and Schacht~\cite{ReiRS18} classified $3$-uniform hypergraphs with uniform Tur\'an density equal to $0$.
The first, third and fifth authors~\cite{GarKL}
constructed a family of $3$-uniform hypergraphs with uniform Tur\'an density equal to $1/27$.
Finally,
Buci\'c, Cooper, Mohr, Munh{\'a} Correia and the third author~\cite{BucCKMM23} determined the uniform Tur\'an density
of $3$-uniform tight cycles of length at least five (the density is equal to $4/27$, if the length is not divisible by three, and
it is $0$ otherwise).
Lastly, the approach from~\cite{ReiRS18a} was extended to a broader family of hypergraphs with uniform Tur\'an density equal to $1/4$~\cite{CheS22,LiLWZ23}.
Hence there are only four values known to be the uniform Tur\'an density of some hypergraph, and those are $0$, $1/27$, $4/27$ and $1/4$.
For further exposition,
we refer the reader to the survey by Reiher~\cite{Rei20} on uniform Tur\'an densities of hypergraphs,
which also includes results from e.g.~\cite{ReiRS16,ReiRS18b,ReiRS18c} on stronger notions of uniform density.

Our main result is an easy to verify sufficient condition for a $3$-uniform hypergraph $H$ to have uniform Tur\'an density equal to $8/27$.
We then identify $3$-uniform hypergraphs $H$ that satisfy this condition, and so
we add $8/27$ to the list of known uniform Tur\'an densities.
Since we mostly deal with $3$-uniform hypergraphs, for the rest of the paper we drop the adjective $3$-uniform when we talk about $3$-uniform hypergraphs except for those places where
we wish to emphasize that a considered hypergraph is $3$-uniform in the interest of clarity (e.g.,
in the statements of lemmas and theorems).

To state our condition, we need the notion of a palette,
which is inspired by the construction of R\"odl from~\cite{Rod86};
we also refer the reader to the survey~\cite{Rei20} for further details on the notion.
A \emph{palette} is a set of ordered triples of colors.
For example, the simplest palette is the palette $\Phi_0=\{(\alpha,\beta,\gamma)\}$,
which consists of a single triple of mutually distinct colors denoted by $\alpha$, $\beta$ and $\gamma$.
More complex examples of palettes are the palettes $\Phi_3$ and $\Phi_8$,
which are defined below before the statement of Theorem~\ref{thm:main}.
If $\Phi$ is a palette,
we say that an $n$-vertex hypergraph $H$ is \emph{$\Phi$-colorable}
if there exists an ordering $v_1,\ldots,v_n$ of the vertices of $H$ and
a coloring of pairs of the vertices of $H$ with colors appearing in the palette $\Phi$ such that
if the vertices $v_i$, $v_j$ and $v_k$, $1\le i<j<k\le n$, form an edge of $H$,
then the triple $(c_{ij},c_{jk},c_{ik})$ is contained in $\Phi$
where $c_{ij}$ is the color of the pair $v_i$ and $v_j$,
$c_{jk}$ is the color of the pair $v_j$ and $v_k$, and
$c_{ik}$ is the color of the pair $v_i$ and $v_k$.

Having defined the notions of a palette and a $\Phi$-colorable hypergraph,
we can state the above mentioned characterization of hypergraphs with the uniform Tur\'an density equal to $0$.

\begin{theorem}[{Reiher, R\"odl and Schacht~\cite{ReiRS18}}]
\label{thm:zero}
The uniform Tur\'an density of a $3$-uniform hypergraph $H$ is equal to $0$ if and only if
$H$ is $\Phi_0$-colorable.
\end{theorem}

Palettes also provide a method for lower bound constructions,
which we present next;
also see the survey~\cite{Rei20}.
We remark that a recent result of the last author~\cite{Lam},
which is discussed in the concluding section,
asserts that colorability by palettes fully determines uniform Tur\'an densities.
To present the lower bound construction method using palettes,
we need an additional definition:
the \emph{density} of a palette $\Phi$,
which consists of ordered triples of colors,
is $|\Phi|/k^3$ where $k$ is the number of colors appearing in the triples of $\Phi$.
For example, the density of the palette $\Phi_0$ is $1/27$.

Fix a palette $\Phi$ with density $d$ and an integer $n$.
Let $k$ be the number of colors appearing in the triples of $\Phi$.
We will construct an $n$-vertex hypergraph $H_n$ with vertices $v_1,\ldots,v_n$ as follows.
Color each pair of vertices uniformly at random with one of the $k$ colors appearing in the triples of $\Phi$ and
include an edge formed by vertices $v_i$, $v_j$ and $v_k$, $1\le i<j<k\le n$,
if the triple $(c_{ij},c_{jk},c_{ik})$ is contained in $\Phi$
where $c_{ij}$ is the color of the pair $v_i$ and $v_j$,
$c_{jk}$ is the color of the pair $v_j$ and $v_k$, and
$c_{ik}$ is the color of the pair $v_i$ and $v_k$.
It is not hard to show that for every $\varepsilon>0$ and $\delta>0$,
there exists an integer $n$ such that the following holds with positive probability:
every subset of $n'\ge\varepsilon n$ vertices of $H_n$ spans at least $(d-\delta)\binom{n'}{3}$ edges.
This construction yields that
the uniform Tur\'an density of any hypergraph that is not $\Phi$-colorable must be at least the density of $\Phi$.
We state this conclusion as a proposition.

\begin{proposition}
\label{prop:lower}
Let $\Phi$ be a palette.
If a $3$-uniform hypergraph $H$ is not $\Phi$-colorable,
then the uniform Tur\'an density of $H$ is at least the density of $\Phi$.
\end{proposition}

Using Proposition~\ref{prop:lower} applied to the palette $\Phi_0$,
we derive from Theorem~\ref{thm:zero} that
the uniform Tur\'an density of any hypergraph $H$ is either $0$ or at least $1/27$ (recall that
the construction of hypergraphs with uniform Tur\'an density equal to $1/27$ was given in~\cite{GarKL},
so this jump value is the best possible).

To state our main result, we need the following two palettes,
which we fix for the rest of the paper (the subscripts denote
the number of triples contained in the palettes).
\begin{align*}
\Phi_3 & = \{(\alpha^1,\beta^1,\omega),(\alpha^2,\omega,\gamma^2),(\omega,\beta^3,\gamma^3)\}\\
\Phi_8 & = \{(x,y,z) \mbox{ such that } x\in\{\beta,\gamma\}, y\in\{\alpha,\gamma\} \mbox{ and } z\in\{\alpha,\beta\}\;\}
\end{align*}
Note that the density of the palette $\Phi_8$ is $8/27$.

We are now ready to state our main result,
which provides a condition that
guarantees that the uniform Tur\'an density of a hypergraph is equal to $8/27$.

\begin{theorem}
\label{thm:main}
The uniform Tur\'an density of every $\Phi_3$-colorable $3$-uniform hypergraph that is not $\Phi_8$-colorable is $8/27$.
\end{theorem}

Since the density of the palette $\Phi_8$ is $8/27$,
Theorem~\ref{thm:main} follows from Proposition~\ref{prop:lower} and
Theorem~\ref{thm:upper}, which is proven in Section~\ref{sec:upper}.
In Section~\ref{sec:construct},
we present a construction of hypergraphs that are $\Phi_3$-colorable but not $\Phi_8$-colorable,
which establishes the existence of hypergraphs with uniform Tur\'an density equal to $8/27$.

\section{Preliminaries}

In this section, we introduce the terminology related to partitioned hypergraphs.
This framework, which is presented by Reiher in the survey~\cite{Rei20},
encapsulates hypergraph regularity arguments
while avoiding some of the technicalities of a direct application of the Hypergraph Regularity Lemma.

An \emph{$N$-partitioned hypergraph $H$} is a $3$-uniform hypergraph such that
its vertex set is partitioned to sets $V_{ij}$, $1\le i<j\le N$, and
for every edge $e$ of $H$, there exist indices $1\le i<j<k\le N$ such that
one vertex of $e$ is from $V_{ij}$, one from $V_{jk}$ and one from $V_{ik}$.
An \emph{$(i,j,k)$-triad} is the set of all edges of $H$ that
have exactly one vertex in each of the sets $V_{ij}$, $V_{jk}$ and $V_{ik}$;
note that each edge of $H$ belongs to exactly one triad.
If $v$ is a vertex of $V_{ij}$,
we write $d_{ij\to k}(v)$ for the number of edges of the $(i,j,k)$-triad that contain $v$ divided by $|V_{ik}|\cdot|V_{jk}|$;
we use $d_{jk\to i}(v)$ and $d_{ik\to j}(v)$ analogously.
Finally,
the \emph{density} of an $(i,j,k)$-triad is the number of edges forming the triad divided by $|V_{ij}|\cdot|V_{ik}|\cdot|V_{jk}|$, and
the \emph{density} of an $N$-partitioned hypergraph $H$ is the minimum density of a triad of $H$.

An $N$-partitioned hypergraph $H_0$ with parts $V_{ij}$, $1\le i<j\le N$,
\emph{embeds} an $n$-vertex hypergraph $H$ with vertices $v_1,\ldots,v_n$
if it is possible to choose distinct $1\le a_1,\ldots,a_n\le N$ corresponding to the vertices of $H$ and
vertices $w_{ij}\in V_{a_ia_j}$ for $1\le i<j\le n$ such that
if vertices $v_i$, $v_j$ and $v_k$ of $H$ form an edge,
then $\{w_{ij},w_{jk},w_{ik}\}$ is an edge in the $(a_i,a_j,a_k)$-triad of $H_0$.
Reiher~\cite{Rei20} provided a general statement that
reduces proving an upper bound on the uniform Tur\'an density of a hypergraph $H$
to embedding $H$ in partitioned hypergraphs.

\begin{theorem}[{Reiher~\cite[Theorem 3.3]{Rei20}}]
\label{thm:reiher}
Let $H$ be a $3$-uniform hypergraph and $d\in [0,1]$.
Suppose that for every $\delta>0$ there exists $N$ such that
every $N$-partitioned hypergraph with density at least $d+\delta$ embeds $H$.
Then, the uniform Tur\'an density of $H$ is at most $d$.
\end{theorem}

\noindent We remark that an extension of this general statement yields that
the uniform Tur\'an density of a hypergraph $H$ and any of its blowups are the same. 

We next state four lemmas given in~\cite[Section 4]{BucCKMM23}, which we cast using our terminology;
the lemmas were implicitly proven in~\cite{ReiRS18} using a direct iterative approach, and
alternative proofs based on Ramsey's Theorem can be found in~\cite{GarKL}.
The lemmas are instances of the following metastatement:
if a sufficiently large set $I_0\subseteq [N]$ is given and
every triad of an $N$-partitioned hypergraph $H$ with indices from $I_0$
has a linear-size set of ``good'' vertices (for instance such a property might be having a high degree in the triad),
then there is a large set of indices $I\subseteq I_0$ such that
there is a universal choice of vertices that are good with respect to all triads with indices from $I$.

The first of the lemmas, which we prove for illustrative purposes, is the instance of the metastatement
in the case when good vertices come from the set $V_{ik}$ of a $(i,j,k)$-triad.

\begin{lemma}[Buci\'c, Cooper, Kr\'al', Mohr and Munh\'a Correia~{\cite[Lemma 4.3]{BucCKMM23}}]
\label{lm:inter_ik}
For every $\varepsilon>0$ and $n$,
there exists $n_0$ such that
the following holds for every $N$-partitioned hypergraph $H$ with $N\ge n_0$,
every $I_0\subseteq\{1,\ldots,N\}$ with $|I_0|\ge n_0$, and
every choice of subsets $W_{ijk}\subseteq V_{ik}$, $i<j<k$, $i,j,k\in I_0$ such that $|W_{ijk}|\ge\varepsilon |V_{ik}|$:
there exist $I\subseteq I_0$ with $|I|\ge n$ and $\gamma_{ik}$, $i<k$, $i,k\in I$, such that
$\gamma_{ik}\in W_{ijk}$ for all $j\in I$ such that $i<j<k$.
\end{lemma}

\begin{proof}
Consider $\varepsilon>0$ and $n$, and
let $n_0$ be the Ramsey number such that 
every $2$-edge-colored $n$-uniform complete hypergraph with $n_0$ vertices
contains a monochromatic complete hypergraph with $K=\max\{n^2-n+1,\lceil n/\varepsilon\rceil+2\}$ vertices.
Fix an $N$-partitioned hypergraph $H$ with $N\ge n_0$,
the set $I_0\subseteq\{1,\ldots,N\}$ such that $|I_0|\ge n_0$, and
the subsets $W_{ijk}\subseteq V_{ik}$, $i<j<k$, $i,j,k\in I_0$ such that $|W_{ijk}|\ge\varepsilon |V_{ik}|$.

We now construct an auxiliary $2$-edge-colored $n$-uniform complete hypergraph with vertex set $I_0$ such that
the edge formed by indices $i_1<\cdots<i_n$, $i_1,\ldots,i_n\in I_0$, is colored blue
if the sets $W_{i_1,i_k,i_n}$, $k=2,\ldots,n-1$, have a non-empty intersection, and
it is colored red otherwise,
i.e., the intersection of the sets $W_{i_1,i_k,i_n}$, $k=2,\ldots,n-1$, is the empty set.
By the choice of $n_0$,
there exists a $K$-element set $J\subseteq I_0$ such that
all edges formed by the elements of $J$ have the same color.
Let $i_1<\cdots<i_K$ be the elements of $J$.

First suppose that the common color of the edges formed by the elements of $J$ is red.
By a simple averaging argument, the set $V_{i_1,i_K}$ contains a vertex that
is contained in at least $\frac{\varepsilon (K-2)|V_{i_1,i_K}|}{|V_{i_1,i_K}|}=\varepsilon (K-2)\ge n$ sets $W_{i_1,i_k,i_K}$, $k\in\{2,\ldots,K-1\}$;
let $x$ be any such vertex of $V_{i_1,i_K}$.
However,
any $n-2$ indices $i_k$, $k\in\{2,\ldots,K-1\}$, such that $x\in W_{i_1,i_k,i_K}$ together with $i_1$ and $i_K$
form a blue edge.
Hence, the common color of the edges formed by the elements of $J$ cannot be red.

We next show that the set $I=\{i_1,i_{n+1},i_{2n+1},\ldots,i_{(n-1)n+1}\}$ has the property from the statement of the lemma.
To do so, we need to find elements $\gamma_{i_a,i_b}$ for all $a,b\in\{1,n+1,\ldots,(n-1)n+1\}$ such that $a<b$.
Fix such $i_a$ and $i_b$.
Since the edge of the auxiliary complete hypergraph
formed by the $(b-a)/n+1\le n$ elements $i_a,i_{a+n},\ldots,i_b$ and
any $n-(b-a)/n-1$ elements among $a_{a+1},\ldots,a_{a+n-1}$ is blue,
the sets $W_{i_a,i_k,i_b}$, $k=a+n,\ldots,b-n$, have a non-empty intersection, and
we set $\gamma_{i_a,i_b}$ to be any element contained in their intersection.
\end{proof}

The next two lemmas are the instance of the metastatement in the cases
when good vertices come from the set $V_{ij}$ and the set $V_{jk}$ of a $(i,j,k)$-triad, respectively.

\begin{lemma}[Buci\'c, Cooper, Kr\'al', Mohr and Munh\'a Correia~{\cite[Lemma 4.4]{BucCKMM23}}]
\label{lm:inter_ij}
For every $\varepsilon>0$ and $n$,
there exists $n_0$ such that
the following holds for every $N$-partitioned hypergraph $H$ with $N\ge n_0$,
every $I_0\subseteq\{1,\ldots,N\}$ with $|I_0|\ge n_0$, and
every choice of subsets $W_{ijk}\subseteq V_{ij}$, $i<j<k$, $i,j,k\in I_0$ such that $|W_{ijk}|\ge\varepsilon |V_{ij}|$:
there exist $I\subseteq I_0$ with $|I|\ge n$ and $\alpha_{ij}$, $i<j$, $i,j\in I$, such that
$\alpha_{ij}\in W_{ijk}$ for all $k\in I$ such that $k>j$.
\end{lemma}

\begin{lemma}[Buci\'c, Cooper, Kr\'al', Mohr and Munh\'a Correia~{\cite[Lemma 4.5]{BucCKMM23}}]
\label{lm:inter_jk}
For every $\varepsilon>0$ and $n$,
there exists $n_0$ such that
the following holds for every $N$-partitioned hypergraph $H$ with $N\ge n_0$,
every $I_0\subseteq\{1,\ldots,N\}$ with $|I_0|\ge n_0$, and
every choice of subsets $W_{ijk}\subseteq V_{jk}$, $i<j<k$, $i,j,k\in I_0$ such that $|W_{ijk}|\ge\varepsilon |V_{jk}|$:
there exist $I\subseteq I_0$ with $|I|\ge n$ and $\beta_{jk}$, $j<k$, $j,k\in I$, such that
$\beta_{jk}\in W_{ijk}$ for all $i\in I$ such that $i<j$.
\end{lemma}

The last lemma is an extension to the setting
when good vertices come from the set $V_{ij}$ of a $(a,i,j)$-triad, a $(i,b,j)$-triad and $(i,j,c)$-triad,
i.e.~which vertices are good depends on all three triads.

We conclude this section with another lemma proven in~\cite[Section 4]{BucCKMM23}.

\begin{lemma}[Buci\'c, Cooper, Kr\'al', Mohr and Munh\'a Correia~{\cite[Lemma 4.2]{BucCKMM23}}]
\label{lm:inter_first}
For every $\varepsilon>0$ and $n$,
there exists $n_0$ such that
the following holds for every $N$-partitioned hypergraph $H$ with $N\ge n_0$,
every $I_0\subseteq\{1,\ldots,N\}$ with $|I_0|\ge n_0$, and
every choice of subsets $W_{ij}^{abc}\subseteq V_{ij}$ where $a,b,c,i,j\in I_0$ and $a<i<b<j<c$
such that $|W_{ij}^{abc}|\ge\varepsilon |V_{ij}|$:
there exist $I\subseteq I_0$ with $|I|\ge n$ and $\omega_{ij}$, $i<j$, $i,j\in I$, such that
$\omega_{ij}\in W_{ij}^{abc}$ for all $a,b,c\in I$ such that $a<i<b<j<c$.
\end{lemma}

\section{Embedding}
\label{sec:upper}

This section is devoted to showing that every sufficiently large hypergraph with uniform density larger than $8/27$
contains any fixed $\Phi_3$-colorable hypergraph $H$ as a subhypergraph.
We start with the following simple lemma, which can be found as \cite[Lemma 6]{GarKL};
we include a short proof for completeness.
Recall that $d_{ij\to k}(v)$ is the number of edges of a $(i,j,k)$-triad that contain $v$ divided by $|V_{ik}|\cdot|V_{jk}|$, and
$d_{jk\to i}(v)$ and $d_{ik\to j}(v)$ are used analogously.

\begin{lemma}
\label{lm:profile}
For every $\varepsilon>0$ and $n$,
there exists $n_0$ such that
the following holds for every $N$-partitioned hypergraph $H$ with $N\ge n_0$:
there exist $I\subseteq\{1,\ldots,N\}$ with $|I|\ge n$ and reals $a$, $b$ and $c$ such that
the following holds for all $i,j,k\in I$ such that $i<j<k$:
\begin{align*}
a & \le \frac{|\{v\in V_{ij}\mbox{ with }d_{ij\to k}(v)\ge\varepsilon\}|}{|V_{ij}|} < a+\varepsilon,\\
b & \le \frac{|\{v\in V_{jk}\mbox{ with }d_{jk\to i}(v)\ge\varepsilon\}|}{|V_{jk}|} < b+\varepsilon,\mbox{ and}\\
c & \le \frac{|\{v\in V_{ik}\mbox{ with }d_{ik\to j}(v)\ge\varepsilon\}|}{|V_{ik}|} < c+\varepsilon.
\end{align*}
\end{lemma}

\begin{proof}
Fix $\varepsilon\in (0,1)$ and $n$, and let $K=\lfloor\varepsilon^{-1}+1\rfloor^3$.
Let $n_0$ be such that
every $K$-edge-colored complete $3$-uniform hypergraph with $n_0$ vertices contains an $n$-vertex monochromatic complete hypergraph;
such $n_0$ exists by Ramsey's Theorem.
Let $H$ be an $N$-partitioned hypergraph with $N\ge n_0$, and set
\begin{align*}
s_{ij\to k} & = \frac{|\{v\in V_{ij}\mbox{ with }d_{ij\to k}(v)\ge\varepsilon\}|}{|V_{ij}|},\\
s_{jk\to i} & = \frac{|\{v\in V_{jk}\mbox{ with }d_{jk\to i}(v)\ge\varepsilon\}|}{|V_{jk}|},\mbox{ and}\\
s_{ik\to j} & = \frac{|\{v\in V_{ik}\mbox{ with }d_{ik\to j}(v)\ge\varepsilon\}|}{|V_{ik}|}.
\end{align*}
We next construct an auxiliary $K$-edge-colored complete $3$-uniform hypergraph $H'$ with $N$ vertices:
the color of the edge formed by the $i$-th, $j$-th and $k$-th vertex of $H'$
is the triple $\left(\lfloor s_{ij\to k}/\varepsilon\rfloor,\lfloor s_{jk\to i}/\varepsilon\rfloor,\lfloor s_{ik\to j}/\varepsilon\rfloor\right)$.
By the choice of $N$, there exist a subset $I\subseteq\{1,\ldots,N\}$ with $|I|=n$ and a triple $(a,b,c)$ such that
all edges formed by the $i$-th, $j$-th and $k$-th vertex of $H'$ have the color $(a,b,c)$ whenever $i,j,k\in I$.
Since the set $I$ and the reals $a\cdot\varepsilon$, $b\cdot\varepsilon$ and $c\cdot\varepsilon$ have the properties given in the statement of the lemma,
the proof of the lemma is finished.
\end{proof}

We are now ready to prove the main result of this section.

\begin{theorem}
\label{thm:upper}
Let $H$ be a $3$-uniform hypergraph.
If $H$ is $\Phi_3$-colorable, then the uniform Tur\'an density of $H$ is at most $8/27$.
\end{theorem}

\begin{proof}
Fix a $\Phi_3$-colorable hypergraph $H$ and let $n$ be the number of vertices of $H$.
By Theorem~\ref{thm:reiher},
it is enough to show that for every $\delta>0$, there exists $N_0$ such that
every $N_0$-partitioned hypergraph with density at least $8/27+\delta$ embeds $H$.
Fix $\delta\in (0,1)$.
We next choose values $N_0,\ldots,N_8$ such that $N_0\gg N_1\gg\cdots\gg N_7\gg N_8$.
To do so, set $\varepsilon=\delta/20$, $N_8=n$, and the values of $N_0,\ldots,N_7$ as follows.
\begin{itemize}
\item $N_7$ is the value of $n_0$ from Lemma~\ref{lm:inter_ik} applied with $\varepsilon$ and $N_8$.
\item $N_6$ is the value of $n_0$ from Lemma~\ref{lm:inter_jk} applied with $\varepsilon$ and $N_7$.
\item $N_5$ is the value of $n_0$ from Lemma~\ref{lm:inter_ik} applied with $\varepsilon$ and $N_6$.
\item $N_4$ is the value of $n_0$ from Lemma~\ref{lm:inter_ij} applied with $\varepsilon$ and $N_5$.
\item $N_3$ is the value of $n_0$ from Lemma~\ref{lm:inter_jk} applied with $\varepsilon$ and $N_4$.
\item $N_2$ is the value of $n_0$ from Lemma~\ref{lm:inter_ij} applied with $\varepsilon$ and $N_3$.
\item $N_1$ is the value of $n_0$ from Lemma~\ref{lm:inter_first} applied with $\varepsilon$ and $N_2$.
\item $N_0$ is the value of $n_0$ from Lemma~\ref{lm:profile} applied with $2\varepsilon$ and $N_1$.
\end{itemize}
For this value of $N_0$,
we will show that every $N_0$-partitioned hypergraph with density at least $8/27+\delta$ embeds $H$.

Fix an $N_0$-partitioned hypergraph $H_0$ with density at least $8/27+\delta$.
By Lemma~\ref{lm:profile},
there exist $I_1\subseteq\{1,\ldots,N_0\}$ with $|I_1|\ge N_1$ and reals $a$, $b$ and $c$ such that
\begin{align*}
a & \le \frac{|\{v\in V_{ij}\mbox{ with }d_{ij\to k}(v)\ge 2\varepsilon\}|}{|V_{ij}|} < a+2\varepsilon,\\
b & \le \frac{|\{v\in V_{jk}\mbox{ with }d_{jk\to i}(v)\ge 2\varepsilon\}|}{|V_{jk}|} < b+2\varepsilon,\mbox{ and}\\
c & \le \frac{|\{v\in V_{ik}\mbox{ with }d_{ik\to j}(v)\ge 2\varepsilon\}|}{|V_{ik}|} < c+2\varepsilon,
\end{align*}
for all $i,j,k\in I_1$ and $i<j<k$.

We next show that $a+b+c$ is at least $2+3\varepsilon$,
which we later show to imply that for every choice of $k<i<k'<j<k''$ from $I_1$
the set $V_{ij}$ contains linearly many vertices $v$ such that
$d_{ij\to k''}(v)\ge 2\varepsilon$, $d_{ij\to k}(v)\ge 2\varepsilon$, and $d_{ij\to k'}(v)\ge 2\varepsilon$.
Consider any of the triads, say the $(1,2,3)$-triad.
Let $A$ be the set of all vertices $v\in V_{12}$ with $d_{12\to 3}(v)\ge 2\varepsilon$,
$B$ the set of all vertices $v\in V_{23}$ with $d_{23\to 1}(v)\ge 2\varepsilon$, and
$C$ the set of all vertices $v\in V_{13}$ with $d_{13\to 2}(v)\ge 2\varepsilon$.
Since each vertex of $V_{12}\setminus A$ is in at most $2\varepsilon\cdot |V_{23}|\cdot|V_{13}|$ edges of the $(1,2,3)$-triad,
each vertex of $V_{23}\setminus B$ is in at most $2\varepsilon\cdot |V_{12}|\cdot|V_{13}|$ edges of the $(1,2,3)$-triad, and
each vertex of $V_{13}\setminus C$ is in at most $2\varepsilon\cdot |V_{12}|\cdot|V_{23}|$ edges of the $(1,2,3)$-triad,
the $(1,2,3)$-triad has at most
\[|A|\cdot |B|\cdot |C|+6\varepsilon\cdot |V_{12}|\cdot|V_{13}|\cdot|V_{23}|
  \le \left(abc+12\varepsilon\right)\cdot |V_{12}|\cdot|V_{13}|\cdot|V_{23}| \]
edges.
Since the density of the $N_0$-partitioned hypergraph $H_0$ is at least $8/27+\delta$,
if follows that $abc+12\varepsilon$ is at least $8/27+\delta=8/27+20\varepsilon$ and so $abc\ge 8/27+8\varepsilon$.
Using the Inequality of Arithmetic and Geometric Means, we obtain that
\[a+b+c \ge 3\sqrt[3]{abc} \ge 3\sqrt[3]{8/27+8\varepsilon} \ge 3(2/3+\varepsilon)=2+3\varepsilon.\]

Our next step is to apply Lemma~\ref{lm:inter_first}
to identify vertices $\omega_{ij}$ such that
$d_{ij\to k''}(\omega_{ij})\ge 2\varepsilon$, $d_{ij\to k}(\omega_{ij})\ge 2\varepsilon$, and $d_{ij\to k'}(\omega_{ij})\ge 2\varepsilon$
for every choice of $k<i<k'<j<k''$ from a suitable index set $I_2$;
the vertices $\omega_{ij}$ will correspond to the color $\omega$ from the palette $\Phi_3$.
To apply the lemma, we define sets $W_{ij}^{kk'k''}\subseteq V_{ij}$ for $i,j,k,k',k''\in I_1$ such that $k<i<k'<j<k''$.
Consider such $k<i<k'<j<k''$.
Let $A$ be the set of all $v\in V_{ij}$ with $d_{ij\to k''}(v)\ge 2\varepsilon$,
$B$ the set of all vertices $v\in V_{ij}$ with $d_{ij\to k}(v)\ge 2\varepsilon$, and
$C$ the set of all vertices $v\in V_{ij}$ with $d_{ij\to k'}(v)\ge 2\varepsilon$.
Let $m_d$, $d\in\{0,1,2,3\}$, be the number of vertices of $V_{ij}$ contained in exactly $d$ sets among $A$, $B$ and $C$.
Observe that
\[(a+b+c)|V_{ij}|\le |A|+|B|+|C|=m_1+2m_2+3m_3\le 2(m_1+m_2)+3m_3\le 2|V_{ij}|+3m_3\],
which implies that
there are at least $\frac{a+b+c-2}{3}|V_{ij}|\ge\varepsilon|V_{ij}|$ vertices of $V_{ij}$ contained in $A\cap B\cap C$.
We set $W_{ij}^{kk'k''}=A\cap B\cap C$.
Lemma~\ref{lm:inter_first} implies that
there exist $I_2\subseteq I_1$ with $|I_2|\ge N_2$ and $\omega_{ij}\in V_{ij}$, $i,j\in I_2$ and $i<j$, such that
$\omega_{ij}\in W_{ij}^{kk'k''}$ for all $i,j,k,k',k''\in I_2$ such that $k<i<k'<j<k''$.
In particular,
it holds that $d_{ij\to k}(\omega_{ij})\ge 2\varepsilon$ for all $i,j\in I_2$, $i<j$, and $k\in I_2\setminus\{i,j\}$ (regardless
whether $k<i<j$, $i<k<j$ or $i<j<k$).

Our next goal is to identify vertices that will correspond to the colors $\alpha^1$ and $\beta^1$ of the palette $\Phi_3$.
To identify those corresponding to the color $\alpha^1$, we will apply Lemma~\ref{lm:inter_ij}.
Consider $i,j,k\in I_2$ such that $i<j<k$ and
let $W_{ijk}$ be the set of all vertices $v\in V_{ij}$ contained together with $\omega_{ik}$
in at least $\varepsilon |V_{jk}|$ edges of the $(i,j,k)$-triad.
Since the number of edges containing $\omega_{ik}$ in the triad $(i,j,k)$-triad
is at most
\[|W_{ijk}|\cdot |V_{jk}|+\varepsilon |V_{ij}\setminus W_{ijk}|\cdot |V_{jk}|\le |W_{ijk}|\cdot |V_{jk}|+\varepsilon |V_{ij}|\cdot |V_{jk}|\] and
$d_{ik\to j}(\omega_{ik})\ge 2\varepsilon$, we obtain that $W_{ijk}$ contains at least $\varepsilon |V_{ij}|$ vertices.
Lemma~\ref{lm:inter_ij} yields that
there exist $I_3\subseteq I_2$  with $|I_3|\ge N_3$ and $\alpha^1_{ij}\in V_{ij}$, $i,j\in I_3$ and $i<j$, such that
$\alpha^1_{ij}$ and $\omega_{ik}$ are contained together in at least $\varepsilon |V_{jk}|$ edges of the $(i,j,k)$-triad
for all $i,j,k\in I_3$ such that $i<j<k$.

We next apply Lemma~\ref{lm:inter_jk} to identify the vertices corresponding to the color $\beta^1$.
For $i,j,k\in I_3$ such that $i<j<k$,
we set $W_{ijk}$ to be the set of all vertices $v\in V_{jk}$ that
form together with $\alpha^1_{ij}$ and $\omega_{ik}$ an edge of the $(i,j,k)$-triad;
note that $W_{ijk}$ contains at least $\varepsilon |V_{jk}|$ vertices.
Hence, Lemma~\ref{lm:inter_jk} implies that
there exist $I_4\subseteq I_3$  with $|I_4|\ge N_4$ and $\beta^1_{jk}\in V_{jk}$, $j,k\in I_4$ and $j<k$, such that
$\alpha^1_{ij}$, $\beta^1_{jk}$ and $\omega_{ik}$ form an edge of the $(i,j,k)$-triad
for all $i,j,k\in I_4$ such that $i<j<k$.
In particular,
the vertices $\alpha^1_{ij}$, $\beta^1_{jk}$ and $\omega_{ik}$ of the $(i,j,k)$-triad
indeed correspond to the colors $\alpha^1$, $\beta^1$ and $\omega$ of the palette $\Phi_3$.

In the completely analogous way,
we apply Lemmas~\ref{lm:inter_ij} and~\ref{lm:inter_ik} (in this order) to obtain
$I_6\subseteq I_4$ with $|I_6|\ge N_6$,
$\alpha^2_{ij}$, $i,j\in I_6$ and $i<j$, and
$\gamma^2_{ik}$, $i,k\in I_6$ and $i<k$, such that
$\alpha^2_{ij}$, $\omega_{jk}$ and $\gamma^2_{ik}$ form an edge of the $(i,j,k)$-triad
for all $i,j,k\in I_6$ such that $i<j<k$, and
so the vertices $\alpha^2_{ij}$ and $\beta^2_{ij}$ correspond to the colors $\alpha^2$ and $\beta^2$ of the palette $\Phi_3$.
Then, we apply Lemmas~\ref{lm:inter_jk} and~\ref{lm:inter_ik} to obtain
$I_8\subseteq I_6$ with $|I_8|\ge N_8$,
$\beta^3_{jk}$, $j,k\in I_6$ and $j<k$, and
$\gamma^3_{ik}$, $i,k\in I_6$ and $i<k$, such that
$\omega_{ij}$, $\beta^3_{jk}$ and $\gamma^3_{ik}$ form an edge of the $(i,j,k)$-triad
for all $i,j,k\in I_8$ such that $i<j<k$;
again, the vertices $\alpha^3_{ij}$ and $\beta^3_{ij}$ correspond to the colors $\alpha^3$ and $\beta^3$ of the palette $\Phi_3$.

We now argue that the $N_0$-partitioned hypergraph $H_0$ embeds the hypergraph $H$.
By the assumption of the theorem, the hypergraph $H$ is $\Phi_3$-colorable.
Let $v_1,\ldots,v_n$ be the vertices of $H$ listed in the order such that
there exists a choice of $c_{ij}\in\{\omega,\alpha^1,\beta^1,\alpha^2,\gamma^2,\beta^3,\gamma^3\}$, $i,j\in\{1,\ldots,n\}$ such that
if vertices $v_i$, $v_j$ and $v_k$, $1\le i<j<k\le n$, form an edge of $H$,
then $(c_{ij},c_{jk},c_{ik})\in\Phi_3$.
Let $a_i$, $i\in\{1,\ldots,n\}$, be the $i$-th smallest index contained in $I_8$ (note that $I_8$ has at least $N_8=n$ elements).
For $1\le i<j\le n$,
set $w_{ij}$ among the vertices of $V_{ij}$ as follows:
\[w_{ij}=\begin{cases}
         \omega_{a_ia_j} & \mbox{if $c_{ij}=\omega$,}\\
         \alpha^1_{a_ia_j} & \mbox{if $c_{ij}=\alpha^1$,}\\
         \beta^1_{a_ia_j} & \mbox{if $c_{ij}=\beta^1$,}\\
         \alpha^2_{a_ia_j} & \mbox{if $c_{ij}=\alpha^2$,}\\
         \gamma^2_{a_ia_j} & \mbox{if $c_{ij}=\gamma^2$,}\\
         \beta^3_{a_ia_j} & \mbox{if $c_{ij}=\beta^3$, and}\\
         \gamma^3_{a_ia_j} & \mbox{if $c_{ij}=\gamma^3$.}
         \end{cases}\]
The choice of the vertices $\omega_{ij}$, $\alpha^1_{ij}$, $\beta^1_{ij}$, $\alpha^2_{ij}$, $\gamma^2_{ij}$, $\beta^3_{ij}$ and $\gamma^3_{ij}$, $i,j\in I_8$ and $i<j$,
yields that
if vertices $v_i$, $v_j$ and $v_k$, $1\le i<j<k\le n$, form an edge of $H$,
then the vertices $w_{ij}$, $w_{jk}$ and $w_{ik}$ form an edge of the $(a_i,a_j,a_k)$-triad.
We conclude that the $N_0$-partitioned hypergraph $H_0$ embeds $H$ and
the proof of theorem is finished.
\end{proof}

\section{Construction}
\label{sec:construct}

In this section,
we identify examples of $3$-uniform hypergraphs that are $\Phi_3$-colorable but not $\Phi_8$-colorable.
The key step in the construction is the following proposition.
Recall that a hypergraph is \emph{linear} if any two edges have at most one vertex in common.

\begin{proposition}
\label{prop:construct}
Suppose that there exists an $n$-vertex $m$-edge $5$-uniform linear hypergraph such that $n!<(10/9)^m$.
Then, there exists an $n$-vertex $(3m)$-edge $3$-uniform hypergraph that
is $\Phi_3$-colorable but not $\Phi_8$-colorable.
\end{proposition}

Let us first discuss which hypergraphs satisfy the assumption of Proposition~\ref{prop:construct}.
Observe that any edge-maximal $5$-uniform linear hypergraph with $n$ vertices has at least
\[\frac{\binom{n}{5}}{10\binom{n-2}{3}}=\frac{n(n-1)}{200}=\Omega(n^2)\]
edges.
It follows that if $n$ is sufficiently large,
then any edge-maximal $5$-uniform linear hypergraph with $n$ vertices satisfies the assumption of Proposition~\ref{prop:construct}.
In particular,
there exist constants $c<C$ such that
for any $n$ and $m$ with $Cn\log n<m<cn^2$,
there exists an $n$-vertex $m$-edge $5$-uniform linear hypergraph satisfying the assumption of Proposition~\ref{prop:construct}.

An explicit small example of a $5$-uniform hypergraph satisfying the assumption
can be obtained from the $5$-dimensional affine space $\FF_5^5$ of order $5$.
Recall that the vertices of $\FF_5^5$ are $5$-dimensional vectors over the $5$-element field $\FF_5$ and
any five-element set $\{a+x\cdot b,\;x\in\FF_5\}$ where $a,b\in\FF_5^5$ and $b\not=(0,0,0,0,0)$ is a line;
note that each line corresponds to $5\cdot 4=20$ distinct choices of $a$ and $b$ and
every pair of points is contained together in exactly one line.
The vertex set of the sought $5$-uniform hypergraph is formed by the $n=5^5=3\,125$ points of the space and
each of the $m=\frac{n(n-1)}{20}=\frac{1}{10}\binom{n}{2}=488\,125$ lines of the space forms an edge.
It is easy to check that $n!<(10/9)^m$.

We now prove Proposition~\ref{prop:construct}.

\begin{proof}[Proof of Proposition~\ref{prop:construct}]
Fix an $n$-vertex $m$-edge $5$-uniform linear hypergraph $H_0$ such that $n!<(10/9)^m$.
We construct the $n$-vertex $(3m)$-edge $3$-uniform hypergraph $H$ in a random way.
The vertex set of $H$ is the same as that of $H_0$.
In each edge $e$ of $H_0$ choose two vertices $v$ and $v'$ randomly (independently of the other edges) and
include to $H$ as edges the three triples containing $v$, $v'$ and one of the three remaining vertices of $e$;
see Figure~\ref{fig:construct} for an illustration of the construction.
We now verify that the hypergraph $H$ is $\Phi_3$-colorable and
that it is not $\Phi_8$-colorable with positive probability.

\begin{figure}
\begin{center}
\epsfbox{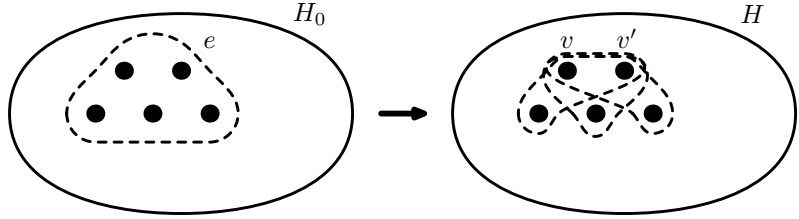}
\end{center}
\caption{The illustration of the construction of the $3$-uniform hypergraph $H$ from the $5$-uniform hypergraph $H_0$.}
\label{fig:construct}
\end{figure}

We start by verifying that the hypergraph $H$ is $\Phi_3$-colorable.
Fix any order $v_1,\ldots,v_n$ of the vertices of $H$ (and so of $H_0$).
If a pair of vertices $v_i$ and $v_j$ of $H_0$ is not contained in an edge of $H_0$, choose $c_{ij}$ arbitrarily.
Otherwise, consider any edge $e$ of $H_0$, and
let $v_i$ and $v_j$, $i<j$, be the two vertices chosen in the construction of $H$.
Set $c_{ij}=\omega$.
For each of the three remaining vertices of $e$, say $v_k$,
set $c_{ik}=\alpha^1$ and $c_{kj}=\beta^1$ if $i<k<j$,
set $c_{ki}=\alpha^2$ and $c_{kj}=\gamma^2$ if $k<i$, and
set $c_{ik}=\beta^3$ and $c_{jk}=\gamma^3$ if $k>j$.
Since the hypergraph $H$ is linear, the values of $c_{ij}$, $1\le i<j\le n$, are well-defined.
It is straightforward to verify that if $\{v_i,v_j,v_k\}$, $1\le i<j<k\le n$, is an edge of $H$,
then $(c_{ij},c_{jk},c_{ik})\in\Phi_3$.
Hence, the hypergraph $H$ is $\Phi_3$-colorable.

We next show that the hypergraph $H$ is not $\Phi_8$-colorable with positive probability.
Fix any order $v_1,\ldots,v_n$ of the vertices of $H$.
Since the $5$-uniform hypergraph $H_0$ is linear,
the hypergraph $H$ is not $\Phi_3$-colorable with respect to the fixed order $v_1,\ldots,v_n$
if and only if
the $5$-uniform hypergraph $H_0$ has an edge $\{v_a,v_b,v_c,v_d,v_e\}$, $1\le a<b<c<d<e\le n$, such that
the two chosen vertices in the construction of $H$ are the vertices $v_b$ and $v_d$.
Indeed, if $H_0$ has such an edge,
then there is no choice of a color $c_{bd}$ such that
$(c_{ab},c_{bd},c_{ad})\in\Phi_8$, $(c_{bc},c_{cd},c_{bd})\in\Phi_8$, and $(c_{bd},c_{de},c_{be})\in\Phi_8$.
And if $H_0$ has no such edge,
we can choose the colors of the pairs of the vertices $v_a,v_b,v_c,v_d,v_e$ within each edge of $H_0$ and
this choice does not affect the choices within other edges of $H_0$ as the hypergraph $H_0$ is linear.
Hence, for any fixed order of the vertices $v_1,\ldots,v_n$,
the probability that
there exists an assignment of colors to the pairs of vertices such that
$(c_{ij},c_{jk},c_{ik})\in\Phi_8$ for every edge $\{v_i,v_j,v_k\}$, $1\le i<j<k\le n$, of $H_0$
is equal to $(9/10)^m$.
The union bound implies that the probability that
there exists a vertex ordering that admits such an assignment of colors to the pairs of vertices
is at most $n!\cdot (9/10)^m<1$,
i.e., the hypergraph $H$ is not $\Phi_8$-colorable with probability at least $1-n!\cdot (9/10)^m>0$.
\end{proof}

\section{Conclusion}
\label{sec:concl}

We are aware of several follow-up results concerning the uniform Tur\'an densities of $3$-uniform hypergraphs, and
we would like to briefly mention the one that we believe to be the most surprising.
Proposition~\ref{prop:lower} asserts that if a hypergraph $H$ is not $\Phi$-colorable for a palette $\Phi$ with density $d$,
then the uniform Tur\'an density of $H$ is at least $d$.
Since all known extremal constructions for uniform Tur\'an densities are based on palette coloring constructions, see~\cite{Rei20},
it is natural to ask whether this is a general phenomenon.
In the subsequent work~\cite{Lam},
the fifth author shows that
for every $\delta>0$
every sufficiently large partitioned hypergraph with density $d$ contains
a partitioned subhypergraph with density $d-\delta$ that can be associated with a palette coloring construction.
It follows that \emph{the uniform Tur\'an density of any $3$-uniform hypergraph $H$
is equal to the supremum over all $d$ such that $H$ is not $\Phi$-colorable for some palette $\Phi$ with density $d$}.
This result brings a powerful tool for determining the uniform Tur\'an densities of hypergraphs;
for example, it was shown that every tight cycle of length at least five
is $\Phi$-colorable for every palette $\Phi$ with density larger than $4/27$ by Cooper in 2018~\cite{Coo18}, %TODO check the title
but the uniform Tur\'an density of tight cycles was determined only 5 years later~\cite{BucCKMM23}.

\bibliographystyle{bibstyle}
\bibliography{turan827}
\end{document}